\documentclass[12pt]{article}%
\usepackage{color}
\usepackage{graphicx}
\usepackage{amsmath}
\usepackage{amsfonts}
\usepackage{amssymb}
\usepackage{enumerate}
\usepackage{amsthm}

\newcommand{\revised}[1]{{#1}}

\theoremstyle{plain}
\newtheorem{theorem}{Theorem}

\newtheorem{corollary}{Corollary}

\newtheorem{lemma}{Lemma}

\newtheorem{proposition}{Proposition}

\theoremstyle{definition}
\newtheorem{definition}{Definition}
\theoremstyle{remark}
\newtheorem{remark}{Remark}

\newcommand{\bfr}{\mathbf{R}}

\newcommand{\bbq}{\mathbb{Q}}

\newcommand{\calM}{{\cal M}}

\newcommand{\ep}{\varepsilon}

\newcommand{\E}{{{\bf E}}}

\newcommand{\rmd}{{\rm{d}}}

\newcommand{\calB}{{\cal B}}
\newcommand{\calb}{{\cal B}}

\newcommand{\calF}{\mathcal{F}}
\newcommand{\calf}{\mathcal{F}}

\newcommand{\one}{\textbf{1}}

\renewcommand{\phi }{\varphi }

\newcommand{\ignore}[1]{}

\usepackage{color}

\usepackage{graphicx}
\usepackage{epsfig}

\begin{document}
\title{Equivalence between Random Stopping Times in Continuous Time%
\thanks{This project grew from a question posed to one of the authors by Yuri Kifer.
Solan acknowledges the support of the Israel Science Foundation, Grants \#212/09 and \#323/13,
and of the Google Inter-university center for Electronic Markets and
Auctions.}}
\author{Eran Shmaya\thanks{Kellogg School of Management, 2001 Sheridan Road, Evanston, IL,
and the School of Mathematical Sciences, Tel Aviv
University, Tel Aviv 69978, Israel. e-mail: e-shmaya@kellogg.northwestern.edu.}\ and
Eilon Solan\thanks{The School of Mathematical Sciences, Tel Aviv
University, Tel Aviv 69978, Israel. e-mail: eilons@post.tau.ac.il.}}
\maketitle

\begin{abstract}
Two concepts of random stopping times in continuous time have been defined in the literature,
mixed stopping times and randomized stopping times.
We show that under weak conditions these two concepts are equivalent, and, in fact, that
all types of random stopping times are equivalent.
We exhibit the significance of the equivalence relation between stopping times using stopping problems and stopping games.
As a by-product we extend Kuhn's Theorem to stopping games in continuous time.
\end{abstract}


\newpage

\section{Introduction}

\revised{
In optimal stopping problems, which have been widely studied in the literature,
a stochastic process is given and the decision maker
has to choose a stopping time to maximize the expectation of
the stopped process.

In some problems in discrete time the optimal stopping time turns out to be randomized;
that is, the decision maker uses a randomization device, which is external to the problem, that dictates when to stop.
Two examples are optimal multivariate stopping problems in which the goal is to maximize
a function of the expectation of the stopped process (Assaf and Samuel-Cahn, 1998)
and the realization of the upper hedging price of an American option in the presence of proportional transaction costs
(Chalasani and Jha, 2001).
Random stopping times have also been introduced for coupling Markov chains, see, e.g., Pitman (1976),
and they arise naturally in game theory, when several decision makers control the stopping time
(see, e.g., Kuhn, 1957).

In continuous time, One definition of random stopping times is that of \emph{mixed} stopping times
(Aumann (1964), Baxter and Chacon (1977)).
Roughly, a mixed stopping time is a probability distribution over stopping times.
Baxter and Chacon (1977) proved that the convex hull of the set of stopping times is the set of mixed stopping times,
and Dalang (1988) and Nualart (1992) proved that
the set of stopping times is the set of extreme points of the set of mixed stopping times.

Another definition of random stopping times is that of \emph{randomized} stopping time
(Touzi and Vieille, 2002).
A randomized stopping time is a non-decreasing $[0,1]$-valued stochastic process that measures,
for each time $t$, the probability that the process is stopped before or at time $t$.

The choice of which definition of random stopping time to use is mainly technical.
For example, Touzi and Vieille (2002) used randomized stopping times because,
after some reductions, this allowed them to use a fixed-point theorem.

A natural question that arises is whether the two types of random stopping times are equivalent.
This question was first asked in the game theory literature,
where it was answered positively for finite games in discrete time
(Kuhn, 1957; see Mertens, Sorin, and Zamir (1994, Chapter II.1.c) for a generalization to infinite games
and Aumann (1964) for a related result).
Recently Tsirelson, Solan, and Vieille (2013) proved that the two concepts are equivalent for discrete-time stopping times.

In this paper we extend Kuhn's Theorem for continuous-time stopping
times, and prove that every mixed stopping time is equivalent to a
unique randomized stopping time, and every randomized stopping time
is equivalent to some (possibly more than one) mixed stopping time.
To this end we introduce the concept of \emph{distribution stopping
time}, that is more general than both types of stopping times
described above, and prove that every distribution stopping time is
equivalent to a unique randomized stopping time and to some mixed
stopping time. Note that in the continuous-time setup it is not
known how to formulate Kuhn's Theorem for other types of games. This
is because the very notion of a strategy in such games, as well as the notion of the play path
generated by a strategy profile, is problematic.

The paper is organized as follows.
In Section \ref{section:definitions} we define the three concepts of random stopping times we are interested in.
In Section \ref{section:equivalent} we define the concept of equivalence between stopping times,
and state and prove the equivalence results.
The motivation for the particular definition of equivalence that we chose is exhibited in the last two sections.
In Section \ref{section:problems} we show that two equivalent random stopping times induce the same payoff in every stopping problem,
and in Section \ref{section:games} we show that two equivalent random stopping times induce the same payoff in every stopping game,
when the strategies of the other players remain fixed.
}

\section{Pure and Random Stopping Times}
\label{section:definitions}

In this section we define pure stopping times in a continuous-time framework and several concepts of random stopping times.
Throughout the paper we equip the unit interval $I = [0,1]$ with the $\sigma$-algebra $\calB$ of Borel sets and the Lebesgue measure $\lambda$.
A stochastic process on a probability space $(\Omega,\calF,P)$ is given by a jointly measurable%
\footnote{Nothing in what is said below will change if the time interval is $[0,\infty)$ rather than $[0,T]$.}
function $x:\Omega\times [0,T]\rightarrow \mathbf{R}$.
When we say that a set is measurable w.r.t.~some $\sigma$-algebra
we always mean the completion of this $\sigma$-algebra w.r.t.\ the underlying measure.
We use the notations $x(\omega,t), x_\omega(t), x_t(\omega)$ interchangeably.
\subsection{On Stopping Problem Forms}
\begin{definition}
A \emph{stopping problem form} (in continuous time) $\Gamma$ is
given by a probability space $(\Omega,\calF,P)$ and a filtration
 in continuous time $(\calF_t)_{t \in [0,T]}$.
\end{definition}

When studying stopping problems in continuous time one usually makes the following technical assumptions.
\begin{definition}
The filtration $(\calF_t)_{t \in [0,T]}$ \emph{satisfies the usual
conditions} if
\begin{enumerate}
\item $\calF = \vee_{t \in [0,T]} \calF_t$, \item $\calF_0$
contains all $P$-null sets, and \item the filtration is
\emph{right continuous}: $\calF_t = \cap_{s > t} \calF_s$, for every $t
\in [0,T]$.
\end{enumerate}
\end{definition}

From now on we fix a stopping problem form $(\Omega,\calF,P,(\calF_t)_{t \in [0,T]})$
in which the filtration satisfies the usual conditions.
Recall that a \emph{stochastic process} $(x_t)_{t\in [0,T]}$ is \emph{adapted}
if $x_t$ is $\calf_t$-measurable for every $t \in [0,T]$.

A stopping time is a $[0,T]$-valued function that is measurable with respect to the filtration.
To emphasize the difference between this concept and the three concepts of random stopping times that we will define in the sequel,
we call the former a \emph{pure stopping time}.
\begin{definition}
A \emph{pure stopping time} is a function $\sigma : \Omega \to [0,T]$
that satisfies $\{\omega \in \Omega \colon \sigma(\omega) \le t\} \in \calF_t$ for every $t \in [0,T]$.
\end{definition}
Because the filtration is right continuous, in the definition of stopping time it is sufficient to require that
$\{\omega \in \Omega \colon \sigma(\omega) < t\} \in \calF_t$ for every $t \in [0,T]$.

\subsection{On Random Stopping Times}

In this section we present two types of random stopping times that were defined in the literature.
The first type is that of mixed stopping times \`a la Aumann (1964),
which was used in Laraki and Solan (2005, 2013).
\begin{definition}
\label{de:mixed:a}
A \emph{mixed stopping time} is a measurable function $\mu : \Omega \times I \to [0,T]$
such that for $\lambda$-almost every $r \in I$, the function $\omega \mapsto \mu(\omega,r)$ is a pure stopping time.
\end{definition}

The interpretation of a mixed stopping time is that the player randomly chooses a pure stopping time according to which he stops.
An equivalent definition for the same concept, which is somewhat more natural
and used in Baxter and Chacon (1977), is the following.

\begin{definition}\label{de:mixed:b}
A \emph{mixed stopping time} is a function $\mu : \Omega \times I \to [0,T]$ such that
$\{(\omega,r) \in \Omega\times I \colon \mu(\omega,r) \le t\}$ is $(\calF_t\otimes \calB)$-measurable for every $t \in [0,T]$.
\end{definition}

\begin{lemma}
Definitions \ref{de:mixed:a} and \ref{de:mixed:b} are equivalent.
\end{lemma}

\begin{proof}
Assume first that $\mu$ is a mixed stopping time according to Definition \ref{de:mixed:a},
and fix $t \in [0,T]$.
Since the set $\{(\omega,r) \in \Omega\times I \colon \mu(\omega,r) \le t\}$
is measurable and its $r$-sections are $\calF_t$-measurable \revised{for $\lambda$-almost every $r$},
it follows that the set is $(\calF_t \otimes \calB)$-measurable (see Solan, Tsirelson, and Vieille, 2013, Proposition 1),
so that $\mu$ satisfies Definition~\ref{de:mixed:b}.
For the other direction, assume $\mu$ is a mixed stopping time according to Definition~\ref{de:mixed:b}.
Then for every $q\in \bbq\cap [0,T]$ it holds that $B_q=\{(\omega,r) \in \Omega\times I \colon \mu(\omega,r) \le q\}$
is measurable, and therefore by Fubini's Theorem its $r$-sections are $\calF_q$-measurable for $\lambda$-almost every $r$.
Consider the set of $r$-s of $\lambda$-measure $1$ for which these sections are $\calF_q$-measurable for every $q$.
Then for these $r$-s we have
$\{\omega \in \Omega \colon \mu(\omega,r) < t\} = \bigcup_{q < t}\{\omega \in \Omega \colon \mu(\omega,r) < q\}\in \calF_t$
for every $t \in [0,T]$, and therefore, by right continuity of the filtration, $\mu(\cdot,r)$ is a pure stopping time.
\end{proof}

A second type of random stopping times is that of \emph{randomized stopping times},
which was used by Touzi and Vieille (2002).
\begin{definition}
A \emph{randomized stopping time} is an adapted
$[0,1]$-valued process $\rho = (\rho_t)_{t \in [0,T]}$ with right-continuous nondecreasing paths that satisfies $\rho_T \equiv 1$.
\end{definition}
The interpretation of a randomized stopping time is that it measures, for each $t
\in [0,T]$, the probability to stop before or at time $t$.

\subsection{On Distribution Stopping Times}

In this section we define a third concept of random stopping times,
called \emph{distribution stopping time}, and see that it is more general than the two concepts
of random stopping times defined before.

Denote by $\calM(\Omega \times [0,T])$ the set of probability measures $\delta$ on $\Omega \times [0,T]$.
Every mixed stopping time $\mu$ naturally defines a measure $\delta_\mu \in \calM(\Omega \times [0,T])$ by
\begin{eqnarray}\label{def:delta-mu}
\label{equ1}
\delta_\mu(A \times [0,t]) &:=& (P \otimes \lambda)(\{ (\omega,r) \colon \omega \in A, \mu(\omega,r) \le t\}), \ \ \ \forall t \in [0,T], \forall A \in \calF.
\label{equ2}
\end{eqnarray}
Thus, $\delta_\mu$ is the push-forward of $P\otimes \lambda$ under the map $(\omega,r)\mapsto (\omega,\mu(\omega,r))$.
Note that the marginal distribution of $\delta_\mu$ on $\Omega$ is
$P$.

Fix a measure $\delta \in \calM(\Omega \times [0,T])$
and for every $t \in [0,T]$ denote by $\delta^t$ the measure on $\Omega$ that is given by
\begin{equation}
\label{def:ep-t}
\delta^t(A):=\delta(A\times [0,t]), \ \ \ \forall A \in \calF.
\end{equation}
If the marginal distribution of $\delta$ on $\Omega$ is $P$, then the measure $\delta^t$ is absolutely continuous w.r.t.~$P$,
and therefore the Radon-Nikodym derivative of $\delta^t$ w.r.t.~$P$
exists.
\begin{lemma}\label{le:thelemma}
Let $\mu$ be a mixed stopping time and let $\delta=\delta_\mu$. In the above notation, the
Radon-Nikodym derivative of $\delta^t$ is $\calF_t$-measurable for
every $t \in [0,T]$ and is given by
\begin{equation}\label{def:f-t}
\omega \mapsto f_t(\omega) := \lambda(\{r \colon \mu(\omega,r) \le t\})
\end{equation}
\end{lemma}

\begin{proof}
By definition, the
set $\{(\omega,r) \colon \mu(\omega,r) \le t\}$ is an
$(\calF_t\otimes\calB)$-measurable set, for every $t \in [0,T]$. By
Fubini's Theorem, the function in Eq.~\eqref{def:f-t}
is defined a.s.~and is $\calf_t$-measurable. Moreover, for every \revised{$t\in[0,T]$} it holds that
\begin{eqnarray}\label{eq:mixed-radon}
\delta^t(A)&=&\delta_\mu(A \times [0,t])=(P\otimes \lambda) (\{(\omega,r) \colon \omega\in A,~\mu(\omega,r) \le t\})\\
&=&\int_A \lambda(\{r : \mu(\omega,r) \le t\}) P(\rmd \omega)=\int_A f_t(\omega) P(\rmd\omega),
\end{eqnarray}
where the first equality follows from~\eqref{def:ep-t}, the second from~\eqref{def:delta-mu}, the third from Fubini's Theorem, and the fourth from~\eqref{def:f-t}.

It follows that $f_t$ is the Radon-Nikodym derivative of $\delta^t$ over $P$. Therefore this derivative is $\calf_t$-measurable.
\end{proof}
The discussion above motivates the definition of a more general class of random stopping times.
\begin{definition}
A \emph{distribution stopping time} is a probability measure $\delta \in \calM(\Omega \times [0,T])$
that satisfies the following two properties:
\begin{enumerate}
\item   The marginal distribution of $\delta$ on $\Omega$ is $P$.
\item   For every $t \in [0,T]$, the Radon-Nikodym derivative of
$\delta^t$ w.r.t.~$P$ is $\calF_t$-measurable.
\end{enumerate}
\end{definition}

If we think of the outcome of a stopping problem as a pair consisting of
(a) the state of nature $\omega$ that is chosen and
(b) the time at which one stops,
then a distribution stopping time is a probability distribution over the space of outcomes.

As we have seen, every mixed stopping time naturally defines a distribution stopping time.
It seems natural that every concept of a random stopping time will induce a probability distribution over $\Omega \times [0,T]$,
and therefore a distribution stopping time.
As we now argue, every randomized stopping time $\rho$ \revised{naturally}
defines a distribution stopping time $\delta_\rho$.
Define a probability measure $\delta_\rho \in \calM(\Omega \times [0,T])$ by
\begin{equation}\label{eq:def-delta-rho}
\delta_\rho(A \times [0,t]) := \int_A \rho_t(\omega) P(\rmd\omega), \ \ \ \forall t \in [0,T],\forall A \in \calF.
\end{equation}
Let $\delta=\delta_\rho$. By definition, the marginal distribution of $\delta$ on $\Omega$ is $P$.
For every $t \in [0,T]$,
\[ \delta^t(A) = \delta_\rho(A \times [0,t]) = \int_A \rho_t(\omega) P(\rmd\omega),\ \ \ \forall A \in \calF,\]
and therefore the Radon-Nikodym derivative of $\delta^t$
w.r.t.~$P$ is $\rho_t$, hence $\calF_t$-measurable.

\section{Equivalence between Stopping Times}
\label{section:equivalent}

Our interest in this paper is the determination of when two random stopping times are equivalent.
To this end we define in this section the equivalence relation we are interested in.
In Sections \ref{section:problems} and \ref{section:games} we explain the significance of this choice of definition.

\begin{definition}
\label{def:equivalence}
Let each of $\eta_1$ and $\eta_2$ be a mixed stopping time or a randomized stopping time.
We say that $\eta_1$ and $\eta_2$ are \emph{equivalent} if they define the same distribution stopping time:
$\eta_1 \equiv \eta_2$ if $\delta_{\eta_1} = \delta_{\eta_2}$.
\end{definition}

\begin{corollary}\label{thecorollary}
Let $\mu$ be a mixed stopping time and let $\rho$ be a randomized stopping time. Then $\mu$ and $\rho$ are equivalent if and only if
\[\rho_t(\omega)=\lambda(\{r\colon\mu(\omega,r)\le t\})\text{ a.s.}\]
for every $t \in [0,T]$.
\end{corollary}

\begin{proof}
One has $\delta_\mu=\delta_\rho$ if and only if $\delta_\mu(A\times [0,t])=\delta_\rho(A\times [0,t])$ for every
$A\in \calf$ and every $t \in [0,T]$,
or, equivalently, if $\delta_\mu^t(A)=\delta_\rho^t(A)$ for every $A\in \calf$ and every $t \in [0,T]$.
By Lemma~\ref{le:thelemma} the Radon-Nykodym derivative of $\delta_\mu^t$ over $P$ is
$f_t(\omega) = \lambda(\{r \colon \mu(\omega,r) \le t\})$.
By definition of $\delta_\rho$ the Radon-Nykodim derivative of $\delta_\rho^t$ over $P$ is $\rho_t$.
It follows that $\delta_\mu=\delta_\rho$ if and only if
\[\rho_t(\omega)=\lambda(\{r\colon\mu(\omega,r)\le t\})\text{ a.s.}\]
\end{proof}

We will also say that a mixed stopping time (resp. a randomized stopping time)
is \emph{equivalent} to the distribution stopping time that it defines: $\mu \equiv \delta_\mu$
(resp. $\rho \equiv \delta_\rho$).
The next theorems state Kuhn's Theorem for these two types of random stopping times.

\begin{theorem}
\label{theorem3}
Every randomized stopping time is equivalent to some mixed stopping time.
\end{theorem}

\begin{proof}
Let $\rho=(\rho_t)_{t\in [0,T]}$ be a randomized stopping
time. We will define a mixed stopping time $\mu$ using $\rho$, and
then show that the two stopping times are equivalent.
Let
$\mu:\Omega\times I\rightarrow[0,T]$ be given
by\[\mu(\omega,r):=\min\{t\in [0,T]\colon \rho_t(\omega)\ge r\}.\]
Note that the minimum is attained because $\rho$ has monotone and right-continuous paths.
The set
\begin{equation}\label{eq:set-mu-le-t}
\{(\omega,r)\colon\mu(\omega,r)\le t\}=\{(\omega,r) \colon\rho_t(\omega) \ge r\}
\end{equation}
is $(\calf_t\otimes\calb)$-measurable as the upper-graph of the $\calf_t$-measurable function $\rho_t$.
In particular, $\mu$ is a mixed
stopping time.

We now prove that $\mu$ is equivalent to
$\rho$. Indeed, from~\eqref{eq:set-mu-le-t},
\[\lambda\left(\{r \colon\mu(\omega,r)\le t\}\right)=\lambda\left(\{r\colon\rho_t(\omega)\ge r\}\right)=\rho_t(\omega)\]
for every $\omega\in\Omega$. By Corollary~\ref{thecorollary} it follows that $\mu$ and $\rho$ are equivalent.
\end{proof}


\begin{theorem}\label{theorem2}
Every distribution stopping time is equivalent to a unique (up to indistinguishability) randomized stopping time.
\end{theorem}

\begin{proof}
For uniqueness, note that if $\rho$ and $\rho'$ are two randomized stopping times such that $\delta_\rho=\delta_{\rho'}$,
then it follows from~\eqref{eq:def-delta-rho} that $\int_A \rho_t(\omega)\rmd P(\omega)=\int_A \rho'_t(\omega)\rmd P(\omega)$ for every $t\in [0,T]$
and every $A\in\calf$.
Therefore, $\rho_t(\cdot)=\rho'_t(\cdot)$ a.s.~for every $t \in [0,T]$.
It follows that, almost surely, $\rho_q(\omega)=\rho'_q(\omega)$ for every $q\in \bbq\cap [0,T]$.
Because $\rho$ and $\rho'$ have right continuous paths it follows that almost surely $\rho_t(\omega)=\rho'_t(\omega)$ for every $t \in [0,T]$, as desired.

For existence, fix a distribution stopping time $\delta$.
Let $\nu:\Omega\rightarrow \Delta([0,T])$ be the disintegration of $\delta$ over $\Omega$, so that
\begin{equation}\label{eq:disi}\E_\delta [R] = \int\left(\int R(\omega,t)\nu_\omega(\rmd t)\right)~P(\rmd\omega)
\end{equation}
for every bounded Borel function $R\colon\Omega\times [0,T]\rightarrow\bfr$.
Set $\rho_t(\omega):=\nu_\omega([0,t])$ for every $t \in [0,T]$ and every $\omega \in \Omega$,
so that for every $\omega\in \Omega$ the function \revised{$t \mapsto \rho_t(\omega)$} is the c.d.f.~
of $\nu_\omega$, and, in particular, nondecreasing and right continuous.
It remains to show that $\delta_\rho=\delta$. And indeed, for every $t\in [0,T]$, we have
\begin{equation}\label{eq:thm-q}
\delta(A\times [0,t])=
\E_\delta[\mathbf{1}_{A \times [0,t]}]=
\int_A\rho_t(\omega) P(\rmd\omega)=\delta_\rho(A\times [0,t])
\end{equation}
where the second equality follows from~\eqref{eq:disi} and the third from~(\ref{eq:def-delta-rho}).
\end{proof}

As a corollary we deduce
that
all types of random stopping times are equivalent.
\begin{corollary}
Suppose that the underlying probability space $(\Omega,\calF,P)$
satisfies the usual conditions.
Then every distribution stopping time is equivalent to a unique (up to indistinguishability)
randomized stopping time and to some mixed stopping time.
\end{corollary}

\begin{remark}
By the definition of equivalence, every equivalence class of stopping times contains a unique distribution stopping time.
By Theorem~\ref{theorem2} this is also true for randomized stopping times.
However, there may be different mixed stopping times that are equivalent.
Indeed, this happens already in the discrete-time framework when $\Omega$ is finite.
Suppose, for example, that $\Omega = \{\omega_1,\omega_2\}$,
\revised{$\calF_0 = 2^\Omega$, and $P$ is the uniform distribution over $\Omega$.}
Consider the distribution stopping time $\delta$ that is the uniform distribution over $\Omega \times \{0,1\}$;
that is, at each state one stops with equal probabilities at times $t=0$ and $t=1$.
There are infinitely many different mixed stopping times that are equivalent to $\delta$, two of them are:
\[ \mu(\omega_i,r) = \left\{
\begin{array}{lll}
0 & \ \ \ \ \ & r \leq \frac{1}{2},\\
1 & & r > \frac{1}{2},
\end{array}
\right.
\ \ \ \ \
\widetilde\mu(\omega_i,r) = \left\{
\begin{array}{lll}
0 & \ \ \ \ \ & i=1, r \leq \frac{1}{2},\\
1 & & i=1, r > \frac{1}{2},\\
0 & \ \ \ \ \ & i=2, r > \frac{1}{2},\\
1 & & i=2, r \leq \frac{1}{2},
\end{array}
\right.
\]
\end{remark}

\section{Equivalence and Stopping Problems}
\label{section:problems}

In this section we provide one motivation to Definition \ref{def:equivalence} of the equivalence relation between random stopping times,
by showing that equivalent random stopping time induce the same payoff in all stopping problems.

 A \emph{stopping problem} (for the filtered probability space
$(\Omega,\calF,P,(\calF_t)_{t \in [0,T]})$)
is given by a bounded%
\footnote{Our results hold for a larger class of payoff processes, namely, the class $%
\mathcal{D}$ that was defined by Dellacherie and Meyer, 1975, \S II-18. This
class contains in particular integrable processes.}
 process $R = (R_t)_{t \in [0,T]}$.
That is, the function $(\omega,t) \mapsto R_t(\omega)$ is $(\calF \times \calB)$-measurable.
For every pure stopping time $\sigma$ denote $R_\sigma(\omega):=R_{\sigma(\omega)}(\omega)$. Then $R_\sigma$ is a random variable (that is, a measurable function).
The payoff that corresponds to the pure stopping time $\sigma$ is
\[ \gamma(R;\sigma) := \E_P[R_\sigma]. \]

Because a mixed stopping time chooses a \revised{pure} stopping time randomly,
the payoff that corresponds to a mixed stopping time $\mu$ is
\begin{equation}\label{eq:gamma-mu} \gamma(R;\mu) := \E_{P \otimes \lambda}[R_{\mu(\omega,r)}(\omega)]. \end{equation}

The payoff that corresponds to a randomized stopping time $\rho$ is (See Touzi and Vieille, 2002)
\begin{equation}
\label{eq:gamma-rho} \gamma(R;\rho) := \E_P\left[\int_0^T \revised{R_t(\omega) \rmd\rho_t(\omega)}\right].
\end{equation}
In the right-hand side the integral is over $t$ and the expectation is over $\omega$.
The payoff that corresponds to a distribution stopping time $\delta$ is
\[ \gamma(R;\delta) := \E_{\delta}[R_t(\omega)]. \]

The next theorem, which follows from the definitions, states that equivalent random stopping times induce the same payoff in all stopping problems.
\begin{theorem}
\label{theorem:1}
Suppose that the filtration satisfies the usual conditions
and fix a stopping problem $R$.
Then $\gamma(R;\mu) = \gamma(R;\rho)$ whenever $\mu$ and $\rho$ are equivalent.
Moreover,
$\gamma(R;\mu) = \gamma(R;\delta_\mu)$ and $\gamma(R;\rho) = \gamma(R;\delta_\rho)$
for every mixed stopping time $\mu$ and every randomized stopping time $\rho$.
\end{theorem}

\begin{proof}
Fix a mixed stopping time $\mu$.
Because $\delta_\mu$ is the push-forward of $P\otimes \lambda$ under the map $(\omega,r)\mapsto (\omega,\mu(\omega,r))$ it follows that
\[\gamma(R;\delta_\mu)=\E_{\delta_\mu}[R_t(\omega)]=\E_{P\otimes\lambda}[R_{\mu(\omega,r)}(\omega)]=\gamma(R;\mu)\]

Now let $\delta$ be a distribution stopping time and let $\rho$ be the unique randomized stopping time such that $\delta=\delta_\rho$.
It follows from the proof of Theorem~\ref{theorem2} that the function
$t \mapsto \rho_t(\omega)$ is the c.d.f.\ of the conditional distributions $\nu_\omega$ of $\delta$ given $\omega$. Therefore from~\eqref{eq:disi}
\[\gamma(R;\delta)=\E_\delta [R_t(\omega)]=
\E_P \left[\revised{\int_0^T R_t(\omega)\rmd\rho_t(\omega)}\right] = \gamma(R;\rho),\]
as desired.
\end{proof}

\section{Game Equivalence}
\label{section:games}

In this section we show that the result of Section \ref{section:problems} carries to multi-player stopping games in continuous time.
In multi-player stopping games, each player has to decide when to stop,
and the terminal payoff depends on the time in which the first player decides to stop,
as well as on the set of players who decide to stop at that time.
To simplify notations we consider only two-player zero-sum games.
The result, however, holds with the same proof for any number of players and for non-zero-sum games as well.

\begin{definition}
A two-player zero-sum \emph{stopping game} in continuous time is given by a
stopping problem form $(\Omega,\calF,P,(\calF_t)_{t \in
[0,T]})$ and by three payoff processes $X = (X_t)_{t
\in [0,T]}$,  $Y = (Y_t)_{t \in [0,T]}$, and  $Z = (Z_t)_{t \in
[0,T]}$.
\end{definition}

The process $X$ (resp.~$Y$, $Z$) dictates the payoff if Player~1 stops before (resp.~after, together with) Player~2,
the goal of Player~1 is to maximize the expected payoff,
and the goal of Player~2 is to minimize this quantity.

We first show how, given a stopping time for Player 2, the game reduces to a single player problem from the perspective of Player~1.
Consider an equivalent class of stopping times for Player~2 and let $\delta_2$ be the unique distribution stopping time in this class.
Assume that Player~2 uses this stopping time
\revised{and denote by $\calb_T$ the Borel $\sigma$-algebra over $[0,T]$}. Let $\tilde\Omega:=\Omega\times [0,T]$.
The stopping problem faced by Player~1 is given by
(a) the probability space $(\tilde\Omega,\calf\otimes\revised{\calb_T},P\otimes\delta_2)$,
(b) the filtration $(\tilde\calf_t)_{t \in [0,T]}$ is given by
$\tilde\calf_t = \{A\times [0,T]\colon A\in\calf_t\}$ for every $t \in [0,T]$, and
(c) the payoff process is
\begin{equation}\label{eq:def-r}R_t(\omega,s) = X_t(\omega)\one_{t<s}+Y_t(\omega)\one_{t>s}+Z_t(\omega)\one_{t=s}.
\end{equation}

Any stopping time of Player~1 that is defined on $(\Omega,\calF,P,(\calF_t)_{t \in
[0,T]})$ can be lifted to a stopping time on $(\tilde\Omega,\calF\otimes\revised{\calb_T},P\otimes\delta_2,\tilde\calf_t)$ as follows.
If $\mu:\Omega\times I\rightarrow [0,\infty)$ is a mixed stopping time,
then $\tilde\mu:\revised{\tilde\Omega}\times I\rightarrow [0,\infty)$
is given by $\tilde\mu(\omega,s,r):=\mu(\omega,r)$.
If $\rho_t$ is a randomized stopping time, then $\tilde\rho_t(\omega,s):=\rho_t(\omega)$.
The following proposition states that this definition of lifting is compatible with our notion of equivalence between stopping times:
\begin{proposition}Let $\mu$ be a mixed stopping time and let $\rho$ be a randomized stopping time on $(\Omega,\calf_t)$.
If $\rho$ and $\mu$ are equivalent, then $\tilde\rho$ and $\tilde\mu$ are equivalent as random stopping times on $(\tilde\Omega,\tilde\calf_t)$.
\end{proposition}
\begin{proof}
For every $(\omega,s)\in \tilde\Omega$ and every \revised{$t\in [0,T]$ we have}
\[\lambda({r\colon\tilde\mu(\omega,s,r)\le t})=\lambda({r\colon\mu(\omega,r)\le t})=\rho_t(\omega)=\tilde\rho_t(\omega,s)\]
where the first equality follows from the definition of $\tilde\mu$, the second from Corollary~\ref{thecorollary},
and the third from the definition of $\tilde\rho$. Therefore, by Corollary~\ref{thecorollary}, $\tilde\mu$ and $\tilde\rho$ are equivalent.
\end{proof}

When Player~1 plays some random stopping time $\tau$ (mixed, randomized, or distribution) the payoff in the game is defined by
\begin{equation}\label{eq:gamma-game}\gamma(X,Y,Z;\tau,\delta_2):=\gamma(R;\tilde\tau),\end{equation}
where on the right-hand side we have the payoff for a single player decision problem on $\tilde\Omega$ with the payoff process $R$ given by~\eqref{eq:def-r}.

With these definitions of the payoff functions Theorem~\ref{theorem:1}
extends from stopping problems to games.
Indeed, by definition,
\revised{the payoff $\gamma(X,Y,Z;\tau,\delta_2)$ depends on Player~2's random stopping time only through
the equivalence class it belongs to.
In particular, the payoff is the same for equivalent random stopping times of Player~2.}
By Theorem~\ref{theorem:1}, we get the same payoff for equivalent random stopping times of Player~1.

The definition of the payoff \revised{in (\ref{eq:gamma-game}) is unsatisfactory because it is}
asymmetric and assumes the perspective of Player~1.
\revised{The following result gets rid of this asymmetry.}
\begin{proposition}\label{game-mus}
\revised{For every pair of mixed stopping times $\mu^1$ and $\mu^2$,}
\begin{eqnarray}
\label{eq:payoff symmetric}
&\gamma(X,Y,Z;\mu^1,\mu^2)=\\&\E_{P \otimes\lambda\otimes\lambda}[
X_{\mu^1(\omega,r^1)}\mathbf{1}_{\{\mu^1(\omega,r^1) < \mu^2(\omega,r^2)\}}
+
Y_{\mu^2(\omega,r^2)}\mathbf{1}_{\{\mu^1(\omega,r^1) > \mu^2(\omega,r^2)\}}
+
Z_{\mu^1(\omega,r^1)}\mathbf{1}_{\{\mu^1(\omega,r^1) = \mu^2(\omega,r^2)\}}].
\nonumber
\end{eqnarray}
\end{proposition}

\revised{The form in Eq.~(\ref{eq:payoff symmetric}) is the one that is usually used in the study of stopping games.}

\begin{proof}
Let \revised{$\delta=\delta_{\mu^2}$ be the distribution stopping time equivalent to} $\mu^2$. Then
\revised{
\begin{align*}
&\gamma(X,Y,Z;\mu^1,\mu^2)\\
&\ =E_{\delta\otimes\lambda}[R_{\tilde\mu^1(\omega,s,r^1)}(\omega,s)]\\
&\ =E_{\delta\otimes\lambda}[R_{\mu^1(\omega,r^1)}(\omega,s)]\\
&\ =E_{P\otimes\lambda\otimes\lambda}[R_{\mu^1(\omega,r^1)}(\omega,\mu^2(\omega,r^2))]\\
&\ =E_{P \otimes\lambda\otimes\lambda}[
X_{\mu^1(\omega,r^1)}\mathbf{1}_{\mu^1(\omega,r^1) < \mu^2(\omega,r^2)}+
Y_{\mu^2(\omega,r^2)}\mathbf{1}_{\mu^1(\omega,r^1) > \mu^2(\omega,r^2)}+
Z_{\mu^1(\omega,r^1)}\mathbf{1}_{\mu^1(\omega,r^1) = \mu^2(\omega,r^2)}],
\end{align*}}
where the first equality follows from~\eqref{eq:gamma-game} and~\eqref{eq:gamma-mu},
the second from the definition of
$\tilde\mu^1$, the third from the fact that
\[E_\delta[R_{\tilde\mu^1(\omega,s,r^1)}(\omega,s)]
=E_{P\otimes\lambda}[R_{\mu^1(\omega,r^1)}(\omega,\mu^2(\omega,r^2))]\]
for
every \revised{$r^1\in [0,1]$} because the left and right hand side are,
respectively,  the payoffs of Player~2 in a single player decision
problem with payoff $R_{\tilde\mu^1(\omega,s,r^1)}$ when playing
$\delta$ and $\mu^2$, and the fourth from~\eqref{eq:def-r}.
\end{proof}
From the symmetry in the expression for the payoff
in Proposition~\ref{eq:def-r} it follows that if we defined the
payoff by reducing the two-player game to the decision problem from
Player~2's perspective we would have obtained the same payoff.

\end{document}